\documentclass[12pt]{amsart}
\usepackage[margin=1.35in]{geometry}
\usepackage{amscd,amsmath,amsxtra,amsthm,amssymb,stmaryrd,xr,mathrsfs,mathtools,enumerate,commath, comment}
\usepackage{stmaryrd}
\usepackage{multirow}
\usepackage{xcolor}
\usepackage{commath}
\usepackage{comment}
\usepackage{tikz-cd}
\usepackage{longtable} 
\usepackage{pdflscape} 
\usepackage{booktabs}
\usepackage{hyperref}
\definecolor{vegasgold}{rgb}{0.77, 0.7, 0.35}
\definecolor{darkgoldenrod}{rgb}{0.72, 0.53, 0.04}
\definecolor{gold(metallic)}{rgb}{0.83, 0.69, 0.22}
\hypersetup{
 colorlinks=true,
 linkcolor=darkgoldenrod,
 filecolor=brown,      
 urlcolor=gold(metallic),
 citecolor=darkgoldenrod,
 }
\newtheorem{lthm}{Theorem}

\usepackage[all,cmtip]{xy}

\DeclareFontFamily{U}{wncy}{}
\DeclareFontShape{U}{wncy}{m}{n}{<->wncyr10}{}
\DeclareSymbolFont{mcy}{U}{wncy}{m}{n}
\DeclareMathSymbol{\Sh}{\mathord}{mcy}{"58}
\usepackage[T2A,T1]{fontenc}
\usepackage[OT2,T1]{fontenc}

\newtheorem{theorem}{Theorem}[section]
\newtheorem{lemma}[theorem]{Lemma}

\newtheorem*{theorem*}{Theorem}
\newtheorem*{ass*}{Assumption}
\newtheorem{definition}[theorem]{Definition}

\newtheorem{conjecture}[theorem]{Conjecture}
\newtheorem{proposition}[theorem]{Proposition}

\newcommand{\cE}{\mathcal{E}}

\newcommand{\Z}{\mathbb{Z}}

\newcommand{\Q}{\mathbb{Q}}
\newcommand{\F}{\mathbb{F}}

\newcommand{\cC}{\mathcal{C}}

\newcommand{\op}[1]{\operatorname{#1}}

\numberwithin{equation}{section}

\begin{document}

\title[Statistics for Iwasawa invariants of elliptic curves, \rm{III}]{Statistics for Iwasawa invariants of elliptic curves, \rm{III}}

\author[A.~Ray]{Anwesh Ray}
\address[Ray]{Chennai Mathematical Institute, H1, SIPCOT IT Park, Kelambakkam, Siruseri, Tamil Nadu 603103, India}
\email{anwesh@cmi.ac.in}


\maketitle

\begin{abstract}Given a prime $p\geq 5$, a conjecture of Greenberg predicts that the $\mu$-invariant of the $p$-primary Selmer group should vanish for most elliptic curves with good ordinary reduction at $p$. In support of this conjecture, I show that the $5$-primary Iwasawa $\mu$- and $\lambda$-invariants simultaneously vanish for an explicit positive density of elliptic curves $E_{/\mathbb{Q}}$. The elliptic curves in question have good ordinary reduction at $5$, and are ordered by their \emph{height}. The results are proven by leveraging work of Bhargava and Shankar on the distribution of $5$-Selmer groups of elliptic curves defined over $\Q$.
\end{abstract}

\section{Introduction}
\subsection{Motivation and historical context}
\par Let $p$ be a prime number and $K$ be a number field. Denote by $\Z_p:=\varprojlim_n \Z/p^n\Z$ the ring of $p$-adic integers. Set $\op{Cl}_p(K)$ to denote the $p$-primary part of the class group of $K$, and $h_p(K):=\# \op{Cl}_p(K)$. Fix an algebraic closure $\bar{K}$ of $K$, and let $\mu_{p^n}\subset \bar{K}$ be the $p^n$-th roots of unity. Let $K(\mu_{p^\infty})$ denote the infinite cyclotomic extension $\bigcup_{n\geq 1} K(\mu_{p^n})$. The \emph{cyclotomic $\Z_p$-extension} of $K$ is the unique Galois extension $K_\infty/K$ contained in $K(\mu_{p^\infty})$ such that $\op{Gal}(K_\infty/K)$ is isomorphic to $\Z_p$ (as a topological group). The \emph{$n$-th layer} $K_n/K$ is the degree $p^n$-extension of $K$ which is contained in $K_\infty$. There is a beautiful asymptotic growth formula for the $p$-primary class numbers $h_p(K_n)$, as $n$ increases. In greater detail, setting $p^{e_n}=h_p(K_n)$, Iwasawa \cite{Iwasawamain} showed that there exist invariants $\mu_p(K), \lambda_p(K)\in \Z_{\geq 0}$ and $\nu_p(K)\in \Z$ such that for $n\gg 0$, one has that
\[e_n=p^n\mu_p(K)  +n \lambda_p(K)+\nu_p(K).\]
It is a longstanding conjecture due to Iwasawa that $\mu_p(K)=0$ for any number field $K$ and any prime number $p$. When $F/\Q$ is an abelian Galois extension, the conjecture was proven by Ferrero and Washington \cite{ferrerowashington}. 

\par It is natural to investigate generalizations of Iwasawa theory to motives and their Selmer groups. Mazur \cite{mazurinitiation} initiated the Iwasawa theory of elliptic curves (and abelian varieties) with good ordinary reduction at a prime $p>2$. Following seminal work of Rubin, Kato, Greenberg and others, the Iwasawa theory of Galois representations associated to motives came into further prominence. The main algebraic objects of study are the $p$-primary Selmer groups associated to elliptic curves, considered over the cyclotomic $\Z_p$-extension of a number field. When the elliptic curve in question has good ordinary reduction at an odd prime $p$, Kato \cite{Katomodform} and Rubin \cite{RubinBSD} showed that the associated Selmer groups are cotorsion over the Iwasawa algebra. This allows one to associated Iwasawa $\mu$ and $\lambda$-invariants $\mu_p(E)$ and $\lambda_p(E)$ respectively to these Selmer groups. These invariants encode the growth properties of $p$-primary Selmer groups, Mordell--Weil groups and Tate-Shafarevich groups in cyclotomic towers. In this context, there is an outstanding conjecture due to Greenberg. Given an elliptic curve $E_{/\Q}$ and a prime number $p$, set $\bar{\rho}_{E,p}:\op{Gal}(\bar{\Q}/\Q)\rightarrow \op{GL}_2(\F_p)$ to denote the Galois representation on the $p$-torsion of $E(\bar{\Q})$.
\begin{conjecture}[Greenberg]
    Let $p$ be an odd prime number and $E_{/\Q}$ be an elliptic curve with good ordinary reduction at $p$. Suppose that the Galois representation $\rho_{E,p}:\op{Gal}(\bar{\Q}/\Q)\rightarrow \op{GL}_2(\F_p)$ is irreducible. Then, the $\mu$-invariant associated to the $p$-primary Selmer group of $E$ over $\Q_\infty$ is equal to $0$. 
\end{conjecture}
Duke \cite{dukeexceptional} showed that the condition that $\bar{\rho}_{E,p}$ is irreducible is satisfied for a density $1$ set of elliptic curves $E_{/\Q}$, when ordered by \emph{height}. Thus, if the set of elliptic curves with good ordinary reduction at $p$ has positive density, the conjecture of Greenberg predicts that $\mu_p(E)=0$ for a positive density of elliptic curves. I prove this expectation for the prime $p=5$. It is indeed the case that many results in the Iwasawa theory of elliptic curves are valid for $p\geq 5$, and the case $p=3$ is altogether more problematic. This is precisely the reason why the case $p=5$ is studied in this article. 

\subsection{Main result}
\par Before stating our main result, I introduce some notation. Any elliptic curve $E_{/\Q}$ admits a unique globally minimal Weierstrass equation. In other words, $E$ is isomorphic to a unique elliptic curve $E_{A,B}$ with minimal Weierstrass equation 
\[E_{A, B}: y^2=x^3+Ax+B,\] where $A, B$ are integers such that $\op{gcd}(A^3 , B^2)$ is not divisible by $d^{12}$ for any integer $d>1$. Let $\cC$ denote the set of all globally minimal Weierstrass equations over $\Z$. The \emph{height of $E$} is defined as $\op{ht}(E):=\op{max}\{|A|^3, B^2\}$. Given a real number $X>0$, set $\cC(X)$ to denote the set of all globally minimal Weierstrass equations $E_{A,B}$ such that $\op{ht}(E_{A, B})\leq X$. One identifies $\cC(X)$ with the set 
\[
\begin{Bmatrix}
 & \abs{A}\leq \sqrt[3]{X}, \ \abs{B}\leq \sqrt{X}\\
(A,B) \in \Z \times \Z: & \quad 4A^3 +27B^2 \neq 0 \\
& \textrm{for all primes } \ell \textrm{ if } \ell^4 |A, \textrm{then } \ell^6\nmid B 
\end{Bmatrix}.
\]
\begin{lthm}\label{main thm}
    There is a set of elliptic curves $E_{/\Q}$ of \emph{positive density} such that have good ordinary reduction at $5$ for which $\mu_5(E)$ and $\lambda_5(E)$ are $0$. In fact, the lower density of this set is explicitly bounded below as follows
    \[\liminf_{X\rightarrow \infty} \frac{\# \{E\in \cC(X)\mid \mu_5(E)=\lambda_5(E)=0\}}{ \# \cC(X)}\geq \frac{\zeta(10) \times \prod_{\ell\geq 7} \left(1-2 \ell^{-2}+\ell^{-3}\right)}{157286400}.\]
\end{lthm}
For an elliptic curve $E_{/\Q}$ for which the Selmer group $\op{Sel}_5(E/\Q)=0$, an explicit criterion due to Greenberg shows that the $\mu_5$- and $\lambda_5$-invariants vanish when certain local invariants vanish, along with certain local invariants, cf. Proposition \ref{main prop}. On the other hand, Bhargava and Shankar \cite{bhargava5selmer} study the distribution of $5$-Selmer groups. In particular, it is shown that a positive density of such curves $E_{/\Q}$ have $5$-Selmer rank $0$. Our result is proven by adapting their methods to show that Greenberg's criterion is satisfied for an explicit positive density of elliptic curves. 

\subsection{Related work} In \cite{kunduraystats1}, Kundu and I studied related questions for a general prime number $p\geq 5$. The density bounds asserted by Theorems 4.3, 4.4 and 4.5 in \emph{loc. cit.} are expressed in terms of a certain quantity $\mathfrak{d}_p^{(1)}$, which is the proportion of elliptic curves $E/\Q$, ordered by height such that $\Sh(E/\Q)[p^\infty]=0$. There are no general upper bounds known for $\mathfrak{d}_p^{(1)}$, however, there are precise predictions for this quantity due to Delaunay \cite{delaunay}. Thus, in the work \cite{kunduraystats1}, it is shown that for $p\geq 5$, the Iwasawa invariants $\mu_p(E)$ and $\lambda_p(E)$ vanish for a positive density of elliptic curves, provided one assumes Delaunay's heuristic. The main feature of the present work is that although the results are only proven only for $p=5$, they are unconditional. 

\subsection{Organization and methodology} Including the introduction, the article consists of three sections. Section \ref{s 2} is devoted to establishing relevant notation and for developing background on the Iwasawa theory of elliptic curves. In section \ref{s 3}, I prove the main result of this article. First, I introduce what it means for a subset of elliptic curves to be defined by local conditions at all primes, and techniques for computing the densities of such subsets. Then, I describe a suitable subset $\cE$ of $\cC$ defined by local conditions such that if the $5$-Selmer group of an elliptic curve $E$ in $\cE$ vanishes, then $\mu_5(E)=0$ and $\lambda_5(E)=0$. This set has the crucial property that exactly half of the elliptic curves $E\in \cE$ has root number $+1$. Furthermore, it is \emph{large} in the sense of \cite{bhargavashankar3, bhargava5selmer}. The final step is then to use the strategy of Bhargava and Shankar to prove that $\op{Sel}_5(E/\Q)=0$ for at least $3/8$-th of the curves in $\cE$. This step crucially makes use of a result of Dokchitser and Dokchitser \cite{DokDok} (cf. \cite[Theorem 39]{bhargava5selmer}). The density of $\cE$ is computed and thus one obtains the explicit lower bound of Theorem \ref{main thm}. 

\section{Iwasawa theory of elliptic curves}\label{s 2}
\par In this section, I discuss preliminaries on the Iwasawa theory of elliptic curves $E_{/\Q}$ with good ordinary reduction at a prime number $p$. I begin by discussing the properties of Selmer groups associated to elliptic curves defined over number fields. Then I introduce Selmer groups considered over the cyclotomic $\Z_p$-extension of $\Q$, and consider these as modules over the Iwasawa algebra. The Iwasawa $\mu$- and $\lambda$-invariants are then defined. I conclude the section with Proposition \ref{main prop}, which is a numerical criterion that gives conditions for the vanishing of these invariants. For a more comprehensive account on the Iwasawa theory of elliptic curves, see for instance \cite{Greenbergmain, greenbergintro}.

\subsection{Selmer groups associated to elliptic curves}
\par Let $p\geq 5$ be a prime number and $E$ be an elliptic curve defined over $\Q$. Let $F/\Q$ be a number field, and choose an algebraic closure $\bar{F}$ of $F$. Denote by $E(F)$ (resp. $\Sh(E/F)$) the Mordell--Weil group (Tate--Shafarevich group) of $E$ over $F$. Set $\Sigma_F$ to be the set of nonarchimedian primes of $F$; for $v\in F$, let $\bar{F}_v/F_v$ be an algebraic closure. Given a natural number $n$, I set $E[n]$ to be the $n$-torsion subgroup of $E(\bar{F})$. Set $\op{G}_F$ (resp. $\op{G}_{F_v}$) to denote the absolute Galois group $\op{Gal}(\bar{\Q}/F)$ (resp. $\op{Gal}(\bar{F}_v/F_v)$). For each prime $\ell\in \Sigma_\Q$, I choose an embedding $\iota_\ell: \bar{\Q}\hookrightarrow \bar{\Q}_\ell$. This induces an injection of Galois groups $j_\ell: \op{G}_{\Q_\ell}\hookrightarrow \op{G}_{\Q}$. Given a prime $v\in \Sigma_F$, let $\ell$ be the rational prime for which $v|\ell$ and identify $\bar{F}_v$ with $\bar{\Q}_\ell$. For $v\in \Sigma_F$, let $j_v:\op{G}_{F_v}\hookrightarrow \op{G}_F$ be the inclusion through which $j_\ell$ factors. This induces a restriction map in cohomology 
\[\op{res}_v: H^1(F, \cdot)\rightarrow H^1(F_v, \cdot).\]
\par Setting $E[p^\infty]:=\bigcup_{n\geq 1} E[p^n]$, denote by $\op{res}_v':H^1(F, E[p^\infty])\rightarrow H^1(F_v, E)[p^\infty]$ the composite 
\[H^1(F, E[p^\infty])\xrightarrow{\op{res}_v} H^1(F_v, E[p^\infty])\rightarrow H^1(F_v, E)[p^\infty].\] Here, the second map arises from the Kummer sequence. Setting $\Phi_F$ to denote the product $\prod_{v\in \Sigma_F} \op{res}_v'$, the $p$-primary Selmer group of $E$ over $F$ is defined as follows
\[\op{Sel}_{p^\infty}(E/F):=\op{ker}\left(H^1(F, E[p^\infty])\xrightarrow{\Phi_F} \prod_v H^1(F_v, E)[p^\infty]\right).\] 
On the other hand, the $p$-Selmer group of $E$ is defined as follows
\[\op{Sel}_{p}(E/F):=\op{ker}\left(H^1(F, E[p])\rightarrow \prod_v H^1(F_v, E)[p]\right).\]
The $p$-primary Selmer group fits into a natural short exact sequence
\[0\rightarrow E(F)\otimes \Q_p/\Z_p\rightarrow \op{Sel}_{p^\infty}(E/F)\rightarrow \Sh(E/F)[p^\infty]\rightarrow 0. \]
The Shafarevich--Tate conjecture predicts that $\Sh(E/F)[p^\infty]$ is finite. In particular, this conjecture implies that 
\[\op{corank}_{\Z_p} \op{Sel}_{p^\infty}(E/F)= \op{rank} E(F). \]

\subsection{Iwasawa theory of Selmer groups}
\par Set $\mu_{p^\infty}:=\bigcup_n \mu_{p^n}$, and $\Q(\mu_{p^\infty}):=\bigcup_n \Q(\mu_{p^n})$. The \emph{cyclotomic $\Z_p$-extension} $\Q_{\infty}/\Q$ is the unique $\Z_p$-extension of $\Q$ which is contained in $\Q(\mu_{p^\infty})$. For $n\geq 0$, the degree $p^n$-extension of $\Q$ contained in $\Q_\infty$ is denoted $\Q_n$, and referred to as the \emph{$n$-th layer} in $\Q_\infty$. Thus one has an infinite tower of extensions 
\[\Q=\Q_0\subset \Q_1\subset \Q_2 \subset \cdots \Q_n \subset \Q_{n+1}\subset \cdots \subset \Q_\infty.\] Set $\Gamma:=\op{Gal}(\Q_\infty/\Q)$ and for $n\geq 0$, identify $\Gamma/\Gamma^{p^n}$ with the Galois group $\op{Gal}(\Q_n/\Q)$. The \emph{Iwasawa algebra} $\Lambda$ is the pro-completed group algebra 
\[\Lambda:=\varprojlim_n \Z_p[\Gamma/\Gamma^{p^n}]. \] The main object of study is the $p$-primary Selmer group $\op{Sel}_{p^\infty}(E/\Q_\infty)$, defined to be the direct limit
\[\op{Sel}_{p^\infty} (E/\Q_\infty):=\varinjlim_n \op{Sel}_{p^\infty} (E/\Q_n)\] with respect to natural restriction maps.
\par Let $\gamma\in \Gamma$ be a topological generator and set $T:=(\gamma-1)\in \Lambda$. With respect to the choice of $\gamma$, I identify $\Lambda$ with the formal power series ring $\Z_p\llbracket T\rrbracket$. A monic polynomial $f(T)\in \Z_p\llbracket T\rrbracket$ is said to be a \emph{distinguished polynomial} if its nonleading coefficients are all divisible by $p$. Let $M$ be a module over $\Lambda$ and set $M^\vee:=\op{Hom}_{\Z_p} \left(M, \Q_p/\Z_p\right)$. Say that $M$ is cofinitely generated (resp. cotorsion) as a $\Lambda$-module if $M^\vee$ is finitely generated (resp. torsion) as a module over $\Lambda$. Suppose that $M$ is cofinitely generated and cotorsion over $\Lambda$. Let $M$ and $M'$ be cofinitely generated and cotorsion $\Lambda$-modules. Then, $M$ and $M'$ are said to be \emph{pseudo-isomorphic} if there is a map $\phi: M\rightarrow M'$ of $\Lambda$-modules, whose kernel and cokernel are finite. By the structure theory of finitely generated and torsion $\Lambda$-modules, any cofinitely generated and cotorsion $\Lambda$-module $M$ is pseudo-isomorphic to a module $M'$ whose Pontryagin dual is given as a direct sum of cyclic torsion $\Lambda$-modules 
\begin{equation}\label{cyclic isomorphism}(M')^\vee\simeq \left(\bigoplus_{i=1}^s \frac{\Lambda}{(p^{n_i})}\right)\oplus \left(\bigoplus_{j=1}^t \frac{\Lambda}{(f_j(T))}\right).\end{equation} Here, $s$ and $t$ are natural numbers (possibly $0$), $n_i \in \Z_{\geq 1}$ and $f_j(T)$ are distinguished polynomials. Then, one defines the Iwasawa $\mu$- and $\lambda$-invariants as follows 
\[\mu_p(M):=\sum_{i=1}^s n_i \text{ and } \lambda_p(M):= \sum_{j=1}^t \op{deg}f_j.\]
Here, if $s=0$ (resp. $t=0$), the sum $\sum_{i=1}^s n_i$ (resp. $\sum_{j=1}^t \op{deg}f_j$) is interpreted as being equal to $0$. These invariants are independent of the choice of $M'$ and the isomorphism \eqref{cyclic isomorphism}; I refer to \cite[Chapter 13]{washingtonintroduction} for further details.
\par Assume that $E$ has good ordinary reduction at $p$. It follows from work of Kato \cite{Katomodform} and Rubin \cite{RubinBSD} that $\op{Sel}_{p^\infty}(E/\Q_\infty)$ is cofinitely generated and cotorsion over $\Lambda$. Set $\mu_p(E)$ (resps. $\lambda_p(E)$) to denote the Iwasawa $\mu$-invariant (resp. $\lambda$-invariant) of the dual Selmer group $\op{Sel}_{p^\infty}(E/\Q_\infty)^\vee$. Denote by $\widetilde{E}$ the reduced curve at $p$, and $\widetilde{E}(\F_p)$ its group of $\F_p$-rational points. The prime $p$ is said to \emph{anomalous} if $p$ divides $\# \widetilde{E}(\F_p)$. For a prime $\ell$, set $c_\ell(E)$ to denote the \emph{Tamagawa number at $\ell$}.

\begin{proposition}[Greenberg]\label{main prop}
    Let $E$ be an elliptic curve and $p\geq 5$ be a prime number of good ordinary reduction for $E$. Assume that the following conditions are satisfied.
    \begin{enumerate}
        \item The Selmer group $\op{Sel}_p(E/\Q)=0$.
        \item The prime $p$ is non-anomalous, i.e., does not divide $\# \widetilde{E}(\F_p)$. 
        \item The prime $p$ does not divide $c_\ell(E)$ for all primes $\ell\neq p$.
    \end{enumerate}
    Then, one has that $\op{Sel}_{p^\infty}(E/\Q_{\infty})=0$. In particular, one has that 
    \[\mu_p(E)=0\text{ and }\lambda_p(E)=0. \]
\end{proposition}
\begin{proof}
    The result follows from \cite[remark following Proposition 3.8, p.80, ll. 19--26]{Greenbergmain}.
\end{proof}

\section{Density results}\label{s 3}
\par In this section I prove the main result of the article. I begin by recalling a density result due to Brumer.
\begin{lemma}\label{brumer lemma}
With respect to notation above, one has the following asymptotic
\[
\# \cC(X) = \frac{4X^{5/6}}{\zeta(10)} + O\left(\sqrt{X}\right).
\]
\end{lemma}
\begin{proof}
    The above result is \cite[Lemma 4.3]{Bru92}.
\end{proof}

\par I discuss the notion of what it means for a subset $\Phi\subseteq \cC$ to be defined by \emph{congruence conditions}. Given a prime $\ell$, a subset $\Phi_\ell\subseteq \cC$ is defined by a congruence condition if there is 
\begin{itemize}
    \item an integer $n_\ell\geq 6$, 
    \item a set of residue classes \[\bar{\Phi}_\ell\subseteq  \{(\bar{A},\bar{B})\in \Z/\ell^{m_\ell}\times \Z/\ell^{m_\ell}\mid \ell^4\nmid \bar{A}\text{ or }\ell^6\nmid \bar{B}\},\]
\end{itemize}
such that $\Phi_\ell$ consists of all pairs $(A, B)\in \cC$ such that \[\left(A\mod{\ell^{m_\ell}},B\mod{\ell^{m_\ell}}\right)\in \bar{\Phi}_\ell.\]
The \emph{local density} of $\Phi_\ell$ is defined as follows\[\mathfrak{d}(\Phi_\ell):=\frac{\# \bar{\Phi}_\ell}{\#\{(\bar{A},\bar{B})\in \Z/\ell^{m_\ell}\times \Z/\ell^{m_\ell}\mid \ell^4\nmid \bar{A}\text{ or }\ell^6\nmid \bar{B}\}}=\frac{\# \bar{\Phi}_\ell}{(1-\ell^{-10})\ell^{2m_\ell}}.\] Let $\ell_1, \dots, \ell_k$ be distinct prime numbers for which there are local conditions $\Phi_{\ell_i}$ defined. Let $m:=\prod_{i=1}^k \ell_i$, set $\Phi_m:=\bigcap_{i=1}^k \Phi_{\ell_i}$ and 
\[\mathfrak{d}(\Phi_m):=\prod_{i=1}^k \mathfrak{d}(\Phi_{\ell_i}). \]
A set $\Phi\subseteq \cC$ is said to be \emph{defined by finitely many congruence conditions} if $\Phi=\Phi_m=\bigcap_{\ell|m} \Phi_{\ell}$ for some squarefree number $m$. More generally, $\Phi\subset \cC$ is simply defined by local conditions if $\Phi=\bigcap_\ell \Phi_\ell$, where the intersection runs over all primes $\ell$ and $\Phi_\ell$ is a local condition at $\ell$. Given a real number $X>0$, set \[\Phi(X):=\Phi\cap \cC(X)=\{(A, B)\in \Phi\mid \op{ht}(E_{A,B})\leq X\}.\]
If $\Phi$ is any subset of $\cC$, the density of $\Phi$ is defined to be the following limit
\[\mathfrak{d}(\Phi):=\lim_{X\rightarrow \infty} \frac{\# \Phi(X)}{\#\cC(X)},\] provided the limit exists. The upper and lower densities are defined as follows
\[\begin{split}&\overline{\mathfrak{d}}(\Phi):=\limsup_{X\rightarrow\infty} \frac{\# \Phi(X)}{\#\cC(X)},\\ 
&\underline{\mathfrak{d}}(\Phi):=\liminf_{X\rightarrow\infty} \frac{\# \Phi(X)}{\#\cC(X)}.\end{split}\]

\begin{theorem}[Bhargava--Shankar]
In any family $\Phi$ defined by finitely many congruence conditions on the coefficients $A$ and $B$, the average size of  $\op{Sel}_5(E/\Q)$ is $6$. In greater detail,
 \[\lim_{X\rightarrow \infty} \frac{\left(\sum_{E\in \Phi(X)} \# \op{Sel}_5(E/\Q)\right)}{\# \Phi(X)}=6.\]
\end{theorem}
More generally, it is conjectured that the average size of $\op{Sel}_n(E/\Q)$ is $\sigma(n)=\sum_{d|n} d$, cf. \cite{poonenrains}.

\begin{proposition}\label{density proposition}
    Consider a subset $\Phi=\bigcap_\ell \Phi_\ell$ of $\cC$ is defined by congruence conditions $\Phi_\ell$ of density $\mathfrak{d}(\Phi_\ell)$. With respect to notation above, suppose that 
    \begin{enumerate}
        \item $\sum_{\ell} \left(1-\mathfrak{d}(\Phi_\ell)\right)<\infty$.
        \item The product $\prod_\ell \mathfrak{d}(\Phi_\ell)$ converges.
        \item Let $\Phi_\ell'$ be the complement $\cC\setminus \Phi_\ell$. Then, there is an absolute constant $C>0$ independent of $\ell$ such that 
        \[\Phi_\ell'(X)\leq C\left(1-\mathfrak{d}(\Phi_\ell)\right) X^{\frac{5}{6}}.\]
    \end{enumerate}
     Then, I have that 
    \[\# \Phi(X)\sim \left(\prod_\ell \mathfrak{d}(\Phi_\ell)\right) \frac{ 4 X^{5/6}}{\zeta(10)}.\]
\end{proposition}
\begin{proof}
Given a positive real number $Z$, set $\Phi_{Z}:=\bigcap_{\ell\leq Z} \Phi_\ell$, where $\ell$ ranges over the prime numbers $\leq Z$. Recall from Lemma \ref{brumer lemma} that \[
\# \cC(X) \sim \frac{4X^{5/6}}{\zeta(10)}.
\]Since there are only finitely many local conditions defining $\Phi_Z$, I have that 
 \[\begin{split} \# \Phi_Z(X) & \sim \left(\prod_{\ell \leq Z} \mathfrak{d}(\Phi_\ell)\right)\times  \# \cC(X) \\ & \sim \left(\prod_{\ell \leq Z} \mathfrak{d}(\Phi_\ell)\right)\times\frac{4 X^{5/6}}{\zeta(10)}.\end{split}\]
 Since $\Phi(X)$ is contained in $\Phi_Z(X)$, I find that 
 \begin{equation}\label{limsup equation}\limsup_{X\rightarrow \infty} \frac{\# \Phi(X)}{X^{5/6}}\leq \lim_{X\rightarrow \infty} \frac{\# \Phi_Z(X)}{X^{5/6}}=\left(\prod_{\ell\leq Z} \mathfrak{d}(\Phi_\ell)\right)\times \frac{4}{\zeta(10)}.\end{equation} Thus, taking $Z\rightarrow \infty$, I find that 
 \[\limsup_{X\rightarrow \infty} \frac{\# \Phi(X)}{X^{5/6}}\leq \left(\prod_{\ell\leq Z} \mathfrak{d}(\Phi_\ell)\right)\times \frac{4}{\zeta(10)}.\]
 Since \[\Phi_Z(X)\subseteq \Phi(X)\cup \left(\bigcup_{\ell>Z } \Phi_\ell'(X)\right),\] I find that 
\[\begin{split}\liminf_{X\rightarrow \infty} \frac{\# \Phi(X)}{X^{5/6}}\geq & \lim_{X\rightarrow \infty} \frac{\# \Phi_Z(X)}{X^{5/6}}-C \sum_{\ell>Z}\left(1-\mathfrak{d}(\Phi_\ell)\right),\\ = &\left(\prod_{\ell \leq Z} \mathfrak{d}(\Phi_\ell)\right)\times\frac{4 X^{5/6}}{\zeta(10)} - C \sum_{\ell>Z}\left(1-\mathfrak{d}(\Phi_\ell)\right). \end{split}\]
Letting $Z\rightarrow \infty$, I find that 
\begin{equation}\label{liminf equation}\liminf_{X\rightarrow \infty} \frac{\# \Phi(X)}{X^{5/6}}\geq \left(\prod_{\ell} \mathfrak{d}(\Phi_\ell)\right)\times\frac{4 X^{5/6}}{\zeta(10)}.\end{equation}
From \eqref{limsup equation} and \eqref{liminf equation}, I deduce that
 \[\# \Phi(X)\sim \left(\prod_\ell \mathfrak{d}(\Phi_\ell)\right) \frac{ 4 X^{5/6}}{\zeta(10)}.\]
\end{proof}

\par Given $E_{A,B}\in \cC$, set $\Delta_{A,B}:=4A^3+27B^2$. For all primes $\ell$ define $\Pi_\ell$ to be the set of all pairs $(A, B)\in \cC$ such that $\ell^2\nmid \Delta_{A,B}$. 

\begin{definition}\label{def of calE}
    I define a subset $\cE=\bigcap_\ell \cE_\ell$ of $\cC$ defined by local conditions at all primes. 
    \begin{itemize}
        \item For all primes $\ell\neq 2,5$, I set $\mathcal{E}_\ell:=\Pi_\ell$.
        \item Let $\cE_2$ be set of pairs $(A,B)\in \cC$ such that $A>0$, $A=4A'$, $B=16 B'$, where 
        $A'\equiv 189\mod{256}$ and $B'\equiv \pm 1\mod{256}$.
        \item Define $\cE_5$ to consist of pairs $(A,B)\equiv (1, \pm 1)\mod{5}$. 
    \end{itemize} 
\end{definition}
For $E\in \cE$, let $E_{-1}=E_{A, -B}$ be the quadratic twist of $E$ by $-1$. It is easy to see that $E_{-1}\in \cE$.

\begin{lemma}\label{local properties of calE}
    Let $E=E_{A,B}$ belong to $\cE$. Then the following assertions hold.
    \begin{enumerate}
        \item\label{p1 of local properties of calE} Then $E$ has additive reduction at $2$, and its j-invariant is a $2$-adic unit.
        \item\label{p2 of local properties of calE} Let $\Delta'$ denote the prime to $2$ part of the discriminant of $E$, then, $\Delta'>0$ and $\Delta'\equiv 1\mod{4}$. 
        \item\label{p3 of local properties of calE} The elliptic curves $E$ has good ordinary reduction at $5$. Moreover, the reduction at $5$ is non-anomalous, i.e., $5\nmid \# \widetilde{E}(\F_5)$.
        \item\label{p4 of local properties of calE} The Tamagawa product $\prod_\ell c_\ell(E)$ is not divisible by $5$.

    \end{enumerate}
\end{lemma}
\begin{proof}
    Throughout this proof, I set $\Delta:=\Delta_{A,B}=4A^3+27B^2$. Part (1) is a direct consequence of Tate's algorithm, cf. \cite[step 7, p. 367]{SilvermanAdvancedtopics}. The j-invariant is computed as follows
    \[j(E)=1728 \frac{4 A^3}{4A^3+27B^2}
    = \frac{27 \times 2^{6}\times  A'^3}{(A'^3+27 B'^2)}.\]
    One finds that 
    \[A'^3+27 B'^2\equiv 189^3+27\equiv 64 \mod{256}.\] This implies in particular that $j(E)$ is a $2$-adic unit. For (2), note that \[\Delta=4 A^3+27 B^2 =2^8 (A'^3+27 B'^2),\] and 
    \[A'^3+27 B'^2=2^6(4n+1),\] where $n$ is a positive natural number. Therefore, $\Delta'>0$ and is $\equiv 1\mod{4}$.
    \par For part (3), consider the curves $\mathcal{A}:y^2=x^3+x+1$ and $\mathcal{A}': y^2=x^3+x-1$. Then, according to data on LMFDB, the Fourier coefficients at $5$ for the associated modular form are given by $a_5(\mathcal{A})=-3$, $a_5(\mathcal{A}')=-3$. Thus, I find that 
    \[\# \widetilde{\mathcal{A}}(\F_5), \# \widetilde{\mathcal{A}'}(\F_5)=9.\] Thus, given any elliptic curve $E\in \cE$, one finds that $\#\widetilde{E}(\F_5)=9$. Thus, $E$ has good ordinary reduction and $5\nmid \#\widetilde{E}(\F_5)$.
    \par I now prove part (4). Note that for all primes $\ell\neq 2$, one has that $\ell^2\nmid \Delta$. This implies that $5\nmid c_\ell(E)$ for all primes $\ell\neq 2$, cf. \cite[2nd column of Table on p. 448]{SilvermanAEC}. In fact, the assertion only requires that $\ell^5\nmid \Delta$. Note that since $E$ has additive reduction at $2$, one deduces that $c_2(E)\leq 4$. This implies that $5\nmid \prod_\ell c_\ell(E)$. 
\end{proof}

\begin{proposition}\label{proposition mu=lambda=0}
    Let $E\in \cE$ be an elliptic curve such that $\op{Sel}_{5}(E/\Q)=0$. Then, it follows that $\op{Sel}_{5^\infty}(E/\Q_\infty)=0$. In particular, $\mu_5(E)=\lambda_5(E)=0$. 
\end{proposition}
\begin{proof}
    Setting $p:=5$, the result follows from Lemma \ref{local properties of calE} and Proposition \ref{main prop}. 
\end{proof}

I shall now show that the set of elliptic curves $E\in \cE$ for which $\op{Sel}_5(E/\Q)$ has positive density, and moreover, give a lower bound for this density. Given an elliptic curve $E$, set $\omega(E)$ to denote its root number. Let $\cE^{+}$ (resp. $\cE^-$) be the subset of curves in $\cE$ for which $\omega(E)=1$ (resp. $\omega(E)=-1$). 
\begin{proposition}\label{cE plus minus prop}
Let $E\in \cE$, the following assertions hold
\begin{enumerate}
    \item $\omega(E)=-\omega(E_{-1})$. 
    \item Exactly half of the curves in $\cE$ have root number $+1$ (resp. $-1$). In other words, 
\[\lim_{X\rightarrow \infty} \left(\frac{\# \cE^\pm(X)}{\# \cE(X)}\right)=\frac{1}{2}.\]
\end{enumerate}
\end{proposition}
\begin{proof}
    For $E\in \cE$, it is easy to see that $E_{-1}\in \cE$ as well. Recall from Lemma \ref{local properties of calE} and the conditions defining $\cE_\ell$, that for $E\in \cE$, the following conditions are satisfied
    \begin{itemize}
        \item $E$ and $E_{-1}$ both have additive reduction at $2$ and their $j$-invariants are $2$-adic units. 
        \item $E$ has squarefree discriminant away from $2$, 
        \item $\Delta'\equiv 1\mod{4}$, where $\Delta'$ is the positive odd part of the discriminant of $E$.
    \end{itemize}
It then follows from the above properties that $\omega(E_{-1})=-\omega(E)$, see \cite[p. 618]{bhargavashankar3} for further details. Since the height of $E$ is the same as that of $E_{-1}$, part (2) is a direct consequence of (1).
\end{proof}

\begin{lemma}\label{lemma d(E_l)}The following assertions hold.
    \begin{enumerate}
        \item For $\ell >5$, one has that $(1-\ell^{-10})\mathfrak{d}(\cE_\ell)=(1-2 \ell^{-2}+\ell^{-3})$.
        \item One has that $\# \cE_\ell'(X) <C\ell^{-2} X^{5/6}$ for some absolute constant $C>0$.
         \item One has that $(1-3^{-10})\mathfrak{d}(\cE_3)= (1- 3^{-2})$.
        \item $(1-2^{-10})\mathfrak{d}(\cE_2)=1/4194304$ and $(1-5^{-10})\mathfrak{d}(\cE_5)=2/25$. 
    \end{enumerate}
\end{lemma}
\begin{proof}
    For part (1), I estimate the number of pairs $(a, b)\in \Z/\ell^6\times \Z/\ell^6$ such that $4a^3\not\equiv  27 b^2\mod{\ell^2}$. Since $\ell\neq 2,3$, one can replace $a$ with $3a$ and $b$ with $2b$ and thus are required to estimate the number of such pairs for which $a^3\neq b^2$. Setting $m_\ell:=6$, it follows from \cite[Lemma 3.6]{CremonaSadek} that $\# \bar{\mathcal{E}_\ell}=\ell^{12}\left(1-2 \ell^{-2}+\ell^{-3}\right)$, and therefore, 
    \[\mathfrak{d}(\mathcal{E}_\ell)=\frac{(1-2 \ell^{-2}+\ell^{-3})}{1-\ell^{-10}}.\]
    \par Part (2) proven in \emph{loc. cit.} immediately following Lemma 3.6.
    \par The proofs of parts (3) and (4) are easy and are thus omitted.
\end{proof}

\begin{proposition}\label{prop 3.9}
    With respect to notation above, the following assertions hold. 
    \begin{enumerate}
        \item The set $\cE$ has density equal to 
    \[\mathfrak{d}(\cE)=\zeta(10) \times \prod_{\ell\geq 7} \left(1-2 \ell^{-2}+\ell^{-3}\right)\times (1-3^{-2})\times \frac{1}{52428800}.\]
    \item The densities of $\cE^+$ and $\cE^-$ are given by 
    \[\mathfrak{d}(\cE^+)=\mathfrak{d}(\cE^-)=\zeta(10) \times \prod_{\ell\geq 7} \left(1-2 \ell^{-2}+\ell^{-3}\right)\times (1-3^{-2})\times \frac{1}{104857600}.\]
    \end{enumerate}
    In particular, $\cE^{+}$ and $\cE^-$ have positive density. 
\end{proposition}
\begin{proof}
    Part (1) of the result follows from Proposition \ref{density proposition} and Lemma \ref{lemma d(E_l)}. Part (2) is a direct consequence of part (1) and Proposition \ref{cE plus minus prop}. 
\end{proof}

\begin{theorem}[Bhargava--Shankar]\label{theorem 3.10}
    When all elliptic curves $E\in \cE$ are ordered according to height, the average size of $\op{Sel}_5(E/\Q)$ is $6$. In other words, one has the following limit
    \[\lim_{X\rightarrow\infty} \frac{\sum_{E\in \cE(X)} \# \op{Sel}_5(E/\Q)}{\# \cE(X)}=6.\]
\end{theorem}
\begin{proof}
    For all but finitely many primes, the condition $\cE_\ell$ coincides with $\Pi_\ell$. Thus, the family $\cE$ is \emph{large} in the sense \cite[p.21, ll.-6 to -4]{bhargava5selmer}. The result then follows from Theorem 31 of \emph{loc. cit.}
\end{proof}

I conclude this article by proving the main result.

\begin{proof}[Proof of Theorem \ref{main thm}]
    As explained in the proof of Theorem \ref{theorem 3.10}, the set $\cE$ is \emph{large}, and the average size of the $5$-Selmer group in this set is $6$. It then follows from \cite[Proposition 40]{bhargava5selmer} that $\op{Sel}_5(E/\Q)=0$ for at least $3/8$-th of the curves in $\cE$. Proposition \ref{proposition mu=lambda=0} asserts that for each $E\in \cE$ such that $\op{Sel}_5(E/\Q)=0$, one has that $\mu_5(E)=0$ and $\lambda_5(E)=0$. Therefore, I conclude that 
    \[\liminf_{X\rightarrow \infty} \frac{\# \{E\in \cC(X)\mid \mu_5(E)=\lambda_5(E)=0\}}{\# \# \cC(X)}\geq \frac{3}{8} \mathfrak{d}(\cE).\]
According to Proposition \ref{prop 3.9},
\[\mathfrak{d}(\cE)=\zeta(10) \times \prod_{\ell\geq 7} \left(1-2 \ell^{-2}+\ell^{-3}\right)\times (1-3^{-2})\times \frac{1}{52428800},\] from which the result follows.
\end{proof}

\bibliographystyle{alpha}
\bibliography{references}
\end{document}